\documentclass[11pt,a4paper]{amsart}
\usepackage{graphicx}
\usepackage{amssymb,amsmath}
\usepackage{amsthm}
\usepackage{epstopdf}
\usepackage{mathtools}
\usepackage{slashed}
\usepackage{xstring}
\usepackage{hyperref}
\usepackage{tikz-cd}
\usepackage{parskip}
\newcommand{\infop}[0]{\bar\partial+\bar\partial^*}
\newcommand{\hook}[0]{\lrcorner\,}

\newtheorem{thm}{Theorem}
\newtheorem*{thm*}{Theorem}
\newtheorem{lem}[thm]{Lemma}
\newtheorem{prop}[thm]{Proposition}
\newtheorem*{prop*}{Proposition}
\newtheorem{cor}[thm]{Corollary}
\newtheorem*{cor*}{Corollary}

\theoremstyle{definition}
\newtheorem*{rem}{Remark}
\newtheorem{defn}{Definition}

\numberwithin{thm}{section}
\numberwithin{equation}{section}
\numberwithin{defn}{section}

\title{Cayley deformations of compact complex surfaces}
\author{Kim Moore}
\address{Department of Mathematics, University College London, Gower Street, London, WC1E  6BT}
\email{kim.moore@ucl.ac.uk}

\begin{document}
\begin{abstract}
In this article, we consider Cayley deformations of a compact complex surface in a Calabi--Yau four-fold. We will study complex deformations of compact complex submanifolds of Calabi--Yau manifolds with a view to explaining why complex and Cayley deformations of a compact complex surface are the same. We in fact prove that the moduli space of complex deformations of any compact complex embedded submanifold of a Calabi--Yau manifold is a smooth manifold.
\end{abstract}

\maketitle

\section{Introduction}\label{sec:intro}
Cayley submanifolds are calibrated submanifolds that arise naturally in manifolds with exceptional holonomy $Spin(7)$. Calibrated submanifolds are by construction volume minimising, and hence minimal submanifolds. Cayley submanifolds also have connections to the proposed program of Donaldson-Thomas \cite{MR1634503}, and more recently Donaldson-Segal \cite{MR2893675}, for higher dimensional gauge theory. In fact, it was proved by Tian \cite{MR1745014} that the blow up loci of $Spin(7)$-instantons are closed Cayley currents.

The most abundant source of Cayley submanifolds are two-dimensional complex submanifolds $N$ of Calabi--Yau four-folds $M$. We can deform $N$ both as a Cayley and as a complex submanifold, but do there exist Cayley deformations of $N$ that are not complex deformations? When $N$ is compact, the following result of Harvey and Lawson may be applied.
\begin{prop}[{\cite[II.4 Thm 4.2]{MR666108}}]\label{prop:calibmin}
Let $X$ be a Riemannian manifold with calibration $\alpha$ and let $Y$ be a compact $\alpha$-calibrated submanifold. Let $Y'$ be any other compact submanifold of $X$ homologous to $Y$. Then
\[
 \int_Y \textnormal{vol}_Y\le \int_{Y'}\textnormal{vol}_{Y'},
\]
with equality if, and only if, $Y'$ is also $\alpha$-calibrated.
\end{prop}
So if $N$ is a compact complex surface inside a Calabi--Yau four-fold $M$ and $N'$ is a Cayley deformation of $N$, then $N'$ is certainly homologous to $N$, and since calibrated submanifolds are volume minimising in their homology class, we must have that
\[
 \int_N \text{vol}_N=\int_{N'}\text{vol}_{N'}.
\]
But then Proposition \ref{prop:calibmin} tells us that $N'$ must also be a complex submanifold. 

This proof is very effective, but does not give any geometric intuition as to why a Cayley deformation of $N$ must be a complex deformation. We know from the work of McLean \cite{MR1664890} that if a Cayley submanifold is spin, then the infinitesimal Cayley deformations of the Cayley submanifold can be identified with the kernel of the twisted Dirac operator.

If a complex surface $N$ is spin, then we have the following identifications \cite[pg 82]{MR1777332}
\begin{align*}
 \mathbb{S}_+&\cong (\Lambda^{0,0}N\oplus \Lambda^{0,2}N)\otimes S_k, \\
\mathbb{S}_-&\cong \Lambda^{0,1}N\otimes S_k,
\end{align*}
where $S_k$ is a holomorphic line bundle satisfying $S_k\otimes S_k=\Lambda^{2,0}N$, and in this case the Dirac operator is given by
\[
 \sqrt{2}(\infop).
\]
Motivated by this, but without requiring $N$ to be spin, we will show in Proposition \ref{prop:caylin} that infinitesimal Cayley deformations of $N$ in $M$ can be identified with the kernel of
\[
\infop: C^\infty(\nu^{1,0}_M(N)\oplus \Lambda^{0,2}N\otimes \nu^{1,0}_M(N))\to C^\infty(\Lambda^{0,1}N\otimes \nu^{1,0}_M(N)).
\]
We will use this to deduce a result, Theorem \ref{thm:compactcay}, on the moduli space of Cayley deformations of $N$ in $M$. We will also give a formula for the expected dimension of this moduli space in terms of topological invariants of $N$ in Theorem \ref{thm:index}.

In the second part of this article, motivated by the study of complex deformations of a complex surface inside a Calabi--Yau four-fold, we will prove, in the style of McLean, a result on the moduli space of complex deformations of any compact complex submanifold $N$ of a Calabi--Yau manifold $M$. As we already know from the seminal work of Kodaira \cite[Thm 1]{MR0133841}, we will see that the infinitesimal complex deformations of $N$ can be identified with the kernel of the operator
\[
\bar\partial :C^\infty(\nu^{1,0}_M(N))\to C^\infty(\Lambda^{0,1}N\otimes \nu^{1,0}_M(N)),
\]
which by Dolbeault's theorem can be identified, in the language of Kodaira, with the sheaf cohomology group $H^0(N,\nu^{1,0}_M(N))$. However, we can actually improve on Kodaira's result in this special case -- that is, we can show that while the obstructions do not necessarily vanish, they do not contribute to the moduli space. We prove in Theorem \ref{thm:cxdefs} that the moduli space of complex deformations of a compact complex embedded submanifold in a Calabi--Yau manifold is a smooth manifold of dimension 
\[
2\, \text{dim}_\mathbb{C}\text{Ker }\bar\partial=\text{dim}_\mathbb{R}\text{Ker }\bar\partial.
\]
To apply the argument used to prove this result, embeddedness of the submanifold is crucial. From this result, we deduce that in order to be able deform a compact complex surface $N$ in a Calabi--Yau four-fold $M$ as a Cayley submanifold into something not complex, we must find $v\in C^\infty(\nu^{1,0}_M(N))$, $w\in C^\infty(\Lambda^{0,2}N\otimes \nu^{1,0}_M(N))$ so that
\[
\bar\partial v=-\bar\partial^* w.
\]
However, this can't happen since $N$ is compact, and so infinitesimal complex and Cayley deformations of $N$ in $M$ are the same. In Theorem \ref{thm:maincomp} we use this to prove that the moduli space of Cayley deformations of $N$ in $M$ is a smooth manifold of dimension 
\[
\text{dim }\text{Ker }(\infop)=2\text{dim }\text{Ker }\bar\partial.
\]
\textbf{Layout.} The report is organised as follows. We begin by recalling some definitions and basic facts about $Spin(7)$-manifolds, Calabi--Yau manifolds and Cayley submanifolds in Section \ref{sec:prelim}. We will prove a result on the moduli space of Cayley deformations of a compact complex surface inside a Calabi--Yau four-fold in Section \ref{sec:cay}. Finally, we will apply McLean's method to study complex deformations of a compact complex submanifolds inside Calabi--Yau manifolds in Section \ref{sec:compdef}.

\textbf{Notation and conventions.} We will take all submanifolds to be embedded. On a complex manifold $M$ we denote $\Lambda^{p,q}M:= \Lambda^p T^{*1,0}M\otimes \Lambda^qT^{*0,1}M$ where $T^{*1,0}M$ and $T^{*0,1}M$ denote the holomorphic and antiholomorphic tangent bundles of $M$. For a complex submanifold $N$ in $M$, we denote by $\nu_M(N), \nu^{1,0}_M(N)$ and $\nu^{0,1}_M(N)$ the normal bundle, holomorphic normal bundle and antiholomorphic normal bundle respectively of $N$ in $M$.
\section{Preliminaries}\label{sec:prelim}
Before we begin we will state the definitions that we will use throughout this report. The following definition is based on the one given in Joyce's book \cite[Defn 11.4.2]{MR2292510}.
\begin{defn}\label{defn:spin7}
 Let $(x_1,\dots,x_8)$ be coordinates on $\mathbb{R}^8$ with the Euclidean metric $g_0=dx_1^2+\dots+dx_8^2$. Define a four-form on $\mathbb{R}^8$ by
\begin{align}\nonumber
 \Phi_0:=&dx_{1234}-dx_{1256}-dx_{1278}-dx_{1357}+dx_{1368}-dx_{1458}-dx_{1467} \\\label{eqn:phi0}
-&dx_{2358}-dx_{2367}+dx_{2457}-dx_{2468}-dx_{3456}-dx_{3478}+dx_{5678},
\end{align}
where $dx_{ijkl}:=dx_i\wedge dx_j\wedge dx_k\wedge dx_l$.

Let $M$ be an eight-dimensional oriented manifold. Define for each $p\in M$ the subset $\mathcal{A}_pM\subseteq \Lambda^4T^*_pM$ to be those four-forms $\Phi$ for which there exists an oriented isomorphism $T_pM\to \mathbb{R}^8$ identifying $\Phi$ and $\Phi_0$ given in \eqref{eqn:phi0}, and define the vector bundle $\mathcal{A}M$ to be the vector bundle with fibre $\mathcal{A}_pM$.

A four-form $\Phi$ on $M$ satisfying $\Phi|_p\in \mathcal{A}_pM$ for all $p\in M$ defines a metric $g$ on $M$, using the fact that each tangent space to $M$ is identified with $\mathbb{R}^8$ with the Euclidean metric. We call $(\Phi,g)$ a $Spin(7)$-\emph{structure} on $M$. Let $\nabla$ denote the Levi-Civita connection of $g$. Say that $(\Phi,g)$ is a \emph{torsion-free} $Spin(7)$-structure on $M$ if $\nabla \Phi=0$.

We say that $(M,\Phi,g)$ is a $Spin(7)$-\emph{manifold} if $M$ is an eight-dimensional oriented manifold and $(\Phi,g)$ is a torsion-free $Spin(7)$-structure on $M$.
\end{defn}

By definition, if $(M,\Phi,g)$ is a $Spin(7)$-manifold then $\Phi$ is a calibration on $M$, known as the \emph{Cayley calibration}. An oriented, four-dimensional submanifold $N$ of $M$ is said to be Cayley if 
\[
\Phi|_N=\text{vol}_N,
\]
i.e., $N$ is $\Phi$-calibrated.

\begin{defn}
Let $(M^m,J,\omega')$ be a compact K\"{a}hler manifold with trivial canonical bundle, that is, there exists a nowhere vanishing section $\alpha$ of $K_M:= \Lambda^{m,0}M$ with $\bar\partial\alpha=0$. Then by Yau's proof of the Calabi conjecture, there exists a Ricci-flat K\"{a}hler form $\omega\in [\omega']$. Choose $\Omega\in \Omega^{m,0}(M)$ so that
\begin{equation}\label{eqn:omOm}
\frac{\omega^m}{m!}=\left(\frac{i}{2}\right)^m(-1)^{m(m-1)/2}\Omega\wedge \overline{\Omega}.
\end{equation}
We call $(M,J,\omega,\Omega)$ a Calabi--Yau manifold.
\end{defn}
Given a four-dimensional Calabi--Yau manifold $(M,J,\omega,\Omega)$, we can define a Cayley form on $M$ by
\begin{equation}\label{eqn:caycy}
\Phi=\frac{1}{2}\omega\wedge \omega+\text{Re }\Omega,
\end{equation}
and so by the choice of constant in \eqref{eqn:omOm}, we can view $M$ as a $Spin(7)$-manifold. Moreover, expression \eqref{eqn:caycy} allows us to see directly that complex surfaces and special Lagrangians are Cayley.

We can decompose bundles of forms on $Spin(7)$-manifolds into irreducible representations of $Spin(7)$. The vector bundle $\Lambda^2_7$ defined below will appear frequently in this exposition. The following proposition can be found in Joyce's book \cite[Prop 11.4.4]{MR2292510}.
\begin{prop}\label{prop:spin7decomp}
Let $M$ be a $Spin(7)$-manifold. Then the bundle of two-forms on $M$ admits the following decomposition into irreducible representations of $Spin(7)$:
\begin{align*}
 \Lambda^2M\cong \Lambda^2_7\oplus \Lambda^2_{21},
\end{align*}
where $\Lambda^k_l$ denotes the irreducible representation of $Spin(7)$ on $k$-forms of dimension $l$.
\end{prop}
\begin{rem}
 Given an orthonormal frame for $M$ $\{e_1,\dots,e_8\}$ with coframe $\{e^1,\dots,e^8\}$, we can explicitly define $\Lambda^2_7$. The following expression is taken from \cite[Thm 9.8]{salamon2010notes}. We have that
\begin{align}\label{eqn:l27}
 \Lambda^2_7=\{e^i\wedge e^j-(e_i\hook(e_j\hook\Phi))\,|\,1\le i<j\le8\}.
\end{align}
\end{rem}

The next result allows us to characterise Cayley submanifolds of a $Spin(7)$-manifold $(X,\Phi,g)$ in terms of a four-form that vanishes exactly when restricted to a Cayley submanifold of $X$.
\begin{prop}[{\cite[Lem 10.15]{salamon2010notes}}]\label{prop:tau}
Let $X$ be a real eight-dimensional manifold with $Spin(7)$-structure $(\Phi,g)$. Let $Y$ be an oriented four-dimensional submanifold of $X$. Then $Y$ is a Cayley submanifold of $X$ if, and only if, $\tau|_Y\equiv 0$, where $\tau\in C^\infty(\Lambda^4 X\otimes \Lambda^2_7)$ is defined by, for any vector fields $x,u,v,w$ on $X$
\begin{align*}
 \tau(x,u,v,w)&=\frac{1}{4}\left(\pi_7(\Phi(\,\cdot\, ,u, v, w)\wedge x^\flat)-\pi_7(\Phi(\,\cdot\, ,v, w, x) \wedge u^\flat)\right. \\
&\left.+\pi_7(\Phi(\,\cdot\, ,w, x, u) \wedge v^\flat)-\pi_7(\Phi(\,\cdot\, ,x, u, v) \wedge w^\flat)\right),
\end{align*}
where $\pi_7:\Lambda^2X\to \Lambda^2_7$ is the projection map given by $\pi_7(x^\flat\wedge y^\flat)=\frac{1}{2}(x^\flat \wedge y^\flat+\Phi(x,y,\cdot,\cdot))$ and $\flat$ denotes the musical isomorphism $TX\to T^*X$.

Moreover, if $x,u,v,w$ are orthogonal then
\[
 \tau(x,u,v,w)=\pi_7(\Phi(\,\cdot\, ,u, v, w)\wedge x^\flat).
\]
\end{prop}
Given an orthonormal frame $\{e_1,\dots , e_8\}$ for $X$, we can equivalently write
\begin{equation}\label{eqn:taunice}
\tau=\sum_{i=2}^8(e^i\wedge (e_1\hook\Phi)-e^i\wedge (e_1\hook\Phi))\otimes \pi_7(e^1\wedge e^i).
\end{equation}
\section{Cayley deformations of compact complex surfaces}\label{sec:cay}

\subsection{Deformations as normal vector fields}\label{ss:nvf}
Let $X$ be a manifold with a submanifold $Y$. We say that $Y'$ is a deformation of $Y$ in $X$ if there exists a smooth family of embeddings $\iota_t:Y\to X$ such that $\iota_0(Y)=Y$ and $\iota_1(Y)=Y'$.
\begin{defn}\label{defn:modspcay}
 Let $(X,g,\Phi)$ be a $Spin(7)$-manifold, and let $Y$ be a Cayley submanifold of $X$. Define the \emph{moduli space of Cayley deformations of $Y$}, $\mathcal{M}_\text{Cay}(Y)$, to be the set of  deformations $Y'$ of $Y$ that are Cayley submanifolds of $(X,g,\Phi)$.
\end{defn}
We will identify nearby deformations of $Y$ with small normal vector fields on $Y$. For this we require the tubular neighbourhood theorem. A proof of this result can be found in \cite[IV, Thm 5.1]{MR1931083}.
\begin{thm}[Tubular neighbourhood theorem]\label{thm:cpttubnbhd}
 Let $X$ be a Riemannian manifold and $Y$ be a closed embedded submanifold of $X$. Then there exists an open set $V\subseteq \nu_X(Y)$ containing the zero section and an open set $Y\subseteq T\subseteq X$ such that the exponential map
\[
 \textnormal{exp}|_V:V\to T,
\]
is a diffeomorphism.
\end{thm}\subsection{Identifications of vector bundles}\label{ss:isom}

In this section we construct isomorphisms of vector bundles on a complex surface $N$ in a Calabi--Yau four-fold $M$.

\begin{prop}\label{prop:normal}
 Let $N$ be a two-dimensional complex submanifold of a Calabi--Yau four-fold $M$. Then 
\begin{equation}\label{eqn:normal}
 \nu_M(N)\otimes \mathbb{C}\cong \nu_M^{1,0}(N)\oplus \Lambda^{0,2}N\otimes \nu^{1,0}_M(N),
\end{equation}
where $\nu_M(N)$ denotes the normal bundle of $N$ in $M$ and $\nu^{1,0}_M(N)$ denotes the holomorphic normal bundle of $N$ in $M$.
\end{prop}
\begin{proof}
 Recall that on a complex submanifold we have the following splitting of the complexified normal bundle into holomorphic and antiholomorphic parts
\[
 \nu_M(N)\otimes \mathbb{C}\cong \nu_M^{1,0}(N)\oplus \nu_M^{0,1}(N).
\]
Therefore to prove the proposition, it suffices to show that
\[
 \nu^{0,1}_M(N)\cong \Lambda^{0,2}N\otimes \nu^{1,0}_M(N).
\]
Recalling that by the adjunction formula \cite[Prop 2.2.17]{MR2093043}
\[
K_M|_N\cong \Lambda^{2,0}N\otimes \Lambda^2 \nu^{*1,0}_M(N),
\]
consider the map
\begin{align*}
 \nu^{0,1}_M(N)&\to  \Lambda^{0,2}N\otimes \nu^{1,0}_M(N), \\
v&\mapsto \frac{1}{4}(v\hook \overline{\Omega})^\sharp,
\end{align*}
where $\sharp$ denotes the musical isomorphism $\nu^{*0,1}(N)\to \nu^{1,0}(N)$. It is easy to check that this map is bijective. Its inverse is given by
\begin{align*}
\Lambda^{0,2}N\otimes \nu^{1,0}_M(N)&\to\nu^{0,1}_M(N), \\
\alpha\otimes v&\mapsto -\left[*_N(\alpha\wedge (v\hook \Omega))\right]^\sharp,
\end{align*}
where $*_N$ is the real Hodge star on $N$ and $\sharp:\nu^{*1,0}_M(N)\to \nu^{0,1}_M(N)$ is the musical isomorphism.
\end{proof}

\begin{prop}\label{prop:e}
 Let $N$ be a two-dimensional complex submanifold of a Calabi--Yau four-fold $M$. Denote by $E$ the rank four vector bundle in the splitting 
\[
 \Lambda^2_7|_N= \Lambda^2_+N\oplus E,
\]
where $\Lambda^2_7$ was defined in Proposition \ref{prop:spin7decomp} and $\Lambda^2_+N$ denotes self-dual two-forms on $N$. Then we have that
\begin{equation}\label{eqn:eisom}
 E\otimes \mathbb{C}\cong \Lambda^{0,1}N\otimes \nu^{1,0}_M(N),
\end{equation}
where $\nu^{1,0}_M(N)$ denotes the holomorphic normal bundle of $N$ in $M$.
\end{prop}
\begin{proof}
 Since we have the musical isomorphism $\flat:\nu^{1,0}_M(N)\to \nu^{*0,1}_M(N)$, it suffices to show that
\[
 E\otimes \mathbb{C}\cong \Lambda^{0,1}N\otimes \nu^{*0,1}_M(N).
\]
To see this we will show that the projection map
\[
 \pi_7:\Lambda^2M\to \Lambda^2_7,
\]
given by
\[
 \pi_7(v\wedge w)=\frac{1}{2}\left[v\wedge w+\Phi(v^\sharp,w^\sharp,\,\cdot\,,\,\cdot\,)\right],
\]
is a bijection 
\[
 \Lambda^{0,1}N\otimes \nu^{*0,1}_M(N)\to E\otimes \mathbb{C}.
\]
Let $\omega$ be the Ricci-flat K\"ahler metric on $M$ and choose a holomorphic volume form $\Omega$ so that the Cayley form on $M$ is given by
\[
 \Phi=\frac{1}{2}\omega\wedge \omega+\text{Re }\Omega.
\]
Let $v\otimes w \in \Lambda^{0,1}N\otimes \nu^{*0,1}_M(N)$. Then viewing this as a two-form on $M$, we have that
\[
 \pi_7(v\wedge w)=\frac{1}{2}\left[v\wedge w +\frac{1}{2}\omega\wedge \omega(v^\sharp,w^\sharp,\,\cdot\,,\,\cdot\,)+\frac{1}{2}(\Omega+\overline\Omega)(v^\sharp,w^\sharp,\,\cdot\,,\,\cdot\,)\right].
\]
First note that $v^\sharp$ and $w^\sharp$ are of type $(1,0)$, and so straight away we can eliminate the $\overline{\Omega}$ term. Further, since
\[
 \omega(a,b)=g(Ja,b),
\]
for all vector fields $a$ and $b$ on $M$, we see that
\begin{align*}
 \frac{1}{2}\omega\wedge \omega(v^\sharp,w^\sharp,\,\cdot\,,\,\cdot\,)&=\frac{1}{2}\left[\omega(v^\sharp,w^\sharp)\wedge \omega+\omega \wedge \omega(v^\sharp,w^\sharp)\right. \\
&\left.-\omega(v^\sharp,\,\cdot\,)\wedge \omega(w^\sharp, \,\cdot\,)+\omega(w^\sharp,\,\cdot\,)\wedge \omega(v^\sharp,\,\cdot\,)\right] \\
&=\frac{1}{2}\left[-g(Jv^\sharp,\,\cdot\,)\wedge g(Jw^\sharp,\,\cdot\,)+g(Jw^\sharp,\,\cdot\,)\wedge g(Jv^\sharp,\,\cdot\,)\right] \\
&=\frac{1}{2}\left[g(v^\sharp,\,\cdot\,)\wedge g(w^\sharp,\,\cdot\,)-g(w^\sharp,\,\cdot\,)\wedge g(v^\sharp,\,\cdot\,)\right] \\
&=\frac{1}{2}\left[v\wedge w-w\wedge v\right] \\
&=v\wedge w,
\end{align*}
since $v^\sharp$ and $w^\sharp$ are of type $(1,0)$ and using the definition of the musical isomorphism. So we have shown that
\[
 \pi_7(v\wedge w)=v\wedge w +\frac{1}{4}\Omega(v^\sharp,w^\sharp,\,\cdot\,,\,\cdot\,),
\]
where we notice that the second term lies in $\Lambda^{1,0}N\otimes \nu^{*1,0}_M(N)$ when restricted to $N$. It can be shown similarly that for $v\in \Lambda^{1,0}N$ and $w\in \nu^{*1,0}_M(N)$ that
\[
 \pi_7(v\wedge w)=v\wedge w +\frac{1}{4}\overline{\Omega}(v^\sharp,w^\sharp,\,\cdot\,,\,\cdot\,),
\]
and so we see that if $\sigma\in\Lambda^{1,0}N\otimes \nu^{*1,0}_M(N)$ or $\Lambda^{0,1}N\otimes \nu^{*0,1}_M(N)$ then $\pi_7(\sigma)\in \Lambda^{1,0}N\otimes \nu^{*1,0}_M(N)\oplus\Lambda^{0,1}N\otimes \nu^{*0,1}_M(N)$. In particular, $\pi_7(\Lambda^{1,0}N\otimes \nu^{*1,0}_M(N))=\pi_7(\Lambda^{0,1}N\otimes \nu^{*0,1}_M(N))$.

A similar calculation yields that for all $\sigma_1\in \Lambda^{1,0}N\otimes \nu^{*0,1}_M(N), \sigma_2\in \Lambda^{0,1}N\otimes \nu^{*1,0}_M(N)$
\[
 \pi_7(\sigma_1)=0=\pi_7(\sigma_2),
\]
and therefore to check that $E\otimes\mathbb{C}=\pi_7(\Lambda^{0,1}N\otimes \nu^{*0,1}_M(N))$ it suffices to check that $\pi_7(\Lambda^{0,2}N)+ \pi_7(\Lambda^{2,0}N)+\pi_7(\Lambda^{1,1}N)=\Lambda^2_+N$. But since if $v\wedge w$ is a unit element of $\Lambda^{0,2}N,\Lambda^{2,0}N$ or $\Lambda^{1,1}N$ then
\begin{align*}
 \pi_7(v\wedge w)|_N&=\frac{1}{2}\left[v\wedge w+\frac{1}{2}\omega\wedge \omega(v^\sharp,w^\sharp,\,\cdot\,,\,\cdot\,)|_N\right] \\
&=\frac{1}{2}\left[v\wedge w +\text{vol}_N(v^\sharp,w^\sharp,\,\cdot\,,\,\cdot\,) \right]\\
&=\frac{1}{2}\left[v\wedge w+*_N(v\wedge w)\right],
\end{align*}
this is clear. Therefore $E\otimes \mathbb{C}=\pi_7(\Lambda^{0,1}N\otimes \nu^{*0,1}_M(N))$. The inverse map to $\pi_7$ is given by the projection map
\[
 \pi_{0,1}:E\otimes \mathbb{C}\to \Lambda^{0,1}N\otimes \nu^{*0,1}_M(N).
\]
\end{proof}

\subsection{Constructing an operator}
We can use Proposition \ref{prop:tau} to construct a partial differential operator acting on normal vector fields on a compact complex surface $N$ whose kernel will be precisely the normal vector fields on $N$ that yield Cayley deformations of $N$.
\begin{prop}\label{prop:caypde}
 Let $(M,J,\omega,\Omega)$ be a four-dimensional Calabi--Yau manifold with compact two-dimensional complex submanifold $N$. Let $U$ be the image of $V$ from the tubular neighbourhood theorem \ref{thm:cpttubnbhd} under the isomorphism in Proposition \ref{prop:normal}. For a normal vector field $v$ write $N_v:=\exp_v(N)$. The moduli space of Cayley deformations of $N$ in $M$ is locally homeomorphic to the kernel of the following partial differential operator
\begin{align}\nonumber
 F:C^\infty(U)&\to C^\infty(\Lambda^{0,1}N\otimes \nu^{1,0}_M(N)), \\ \label{eqn:caypde}
v&\mapsto \Psi\left(\pi(*_N\exp^*_{\tilde{v}}(\tau|_{N_{\tilde{v}}}))\right),
\end{align}
where $\tau$ is defined in Proposition \ref{prop:tau},
\begin{equation}\label{eqn:27split}
\Lambda^2_7|_N=\Lambda^2_+N\oplus E,
\end{equation}
with $\pi:\Lambda^2_7|_N\to E$ the projection map, $\tilde{v}$ denotes the normal vector field corresponding to $v$ under the isomorphism given in Proposition \ref{prop:normal} and $\Psi$ denotes the isomorphism given in Proposition \ref{prop:e}.
\end{prop}
\begin{proof}
By Proposition \ref{prop:tau} it is clear that a normal vector field $\tilde{v}$ gives a Cayley deformation of $N$ if, and only if, $*_N\exp^*_{\tilde{v}}(\tau|_{N_{\tilde{v}}})=0$. By Propositions \ref{prop:normal} and Proposition \ref{prop:e}, it remains to show that $\pi(*_N\exp_{\tilde{v}}^*(\tau|_{N_{\tilde{v}}}))=0$ implies that $*_N\exp_{\tilde{v}}^*(\tau|_{N_{\tilde{v}}})=0$. For this we will employ a local argument.
 
At each point of $N$, we can write the tangent space of the (small) deformation $N_{\tilde{v}}$ as a normal graph over the tangent space of $N$. So it suffices to prove this proposition for a normal graph over a Cayley plane in $\mathbb{R}^8$. Suppose that this graph is described by, for $j=1,\dots , 4$
\[
 v_j=e_j+\sum_{i=5}^8\lambda^i_je_i.
\]
To prove the proposition we will suppose that
\begin{equation}\label{eqn:tauproj}
\pi(\tau_{f(y)}(v_1,v_2,v_3,v_4))=0, 
\end{equation}
and show that
\begin{equation}\label{eqn:tauorig}
\tau_{f(y)}(v_1,v_2,v_3,v_4)=0.
\end{equation}
Equation \eqref{eqn:tauproj} gives us the following four equations
\begin{align}\nonumber
 \lambda^1_5+\lambda^2_6+\lambda^3_7+\lambda^4_8-&\sum_{6,7,8}\epsilon_{pqr}\lambda^2_p\lambda^3_q\lambda^4_r-\sum_{5,7,8}\epsilon_{pqr}\lambda^1_p\lambda^3_q\lambda^4_r \\ \nonumber -&\sum_{5,6,8}\epsilon_{pqr}\lambda^1_p\lambda^2_q\lambda^4_r-\sum_{5,6,7}\epsilon_{pqr}\lambda^1_p\lambda^2_q\lambda^3_r=0,\\ \nonumber
\lambda^1_6-\lambda^2_5-\lambda^3_8+\lambda^4_7+&\sum_{5,7,8}\epsilon_{pqr}\lambda^2_p\lambda^3_q\lambda^4_r-\sum_{6,7,8}\epsilon_{pqr}\lambda^1_p\lambda^3_q\lambda^4_r  \\
\nonumber
-&\sum_{5,6,7}\epsilon_{pqr}\lambda^1_p\lambda^2_q\lambda^4_r+\sum_{5,6,8}\epsilon_{pqr}\lambda^1_p\lambda^2_q\lambda^3_r=0, 
\end{align}
\begin{align}
 \nonumber
\lambda^1_7+\lambda^2_8-\lambda^3_5-\lambda^4_6-&\sum_{5,6,8}\epsilon_{pqr}\lambda^2_p\lambda^3_q\lambda^4_r-\sum_{5,6,7}\epsilon_{pqr}\lambda^1_p\lambda^3_q\lambda^4_r \\ \nonumber
+&\sum_{6,7,8}\epsilon_{pqr}\lambda^1_p\lambda^2_q\lambda^4_r+\sum_{5,7,8}\epsilon_{pqr}\lambda^1_p\lambda^2_q\lambda^3_r=0, \\ \label{eqn:taunonlin}
\lambda^1_8-\lambda^2_7+\lambda^3_6-\lambda^4_5+&\sum_{5,6,7}\epsilon_{pqr}\lambda^2_p\lambda^3_q\lambda^4_r-\sum_{5,6,8}\epsilon_{pqr}\lambda^1_p\lambda^3_q\lambda^4_r \\ \nonumber
+&\sum_{5,7,8}\epsilon_{pqr}\lambda^1_p\lambda^2_q\lambda^4_r-\sum_{6,7,8}\epsilon_{pqr}\lambda^1_p\lambda^2_q\lambda^3_r=0,
\end{align}
where $\epsilon_{pqr}$ is skew-symmetric in $p,q,r$ and $\epsilon_{pqr}=1$ when $p<q<r$. Notice that if $\lambda^i_j$ is a linear term, then there will be cubic terms of the form $\pm\lambda^l_p\lambda^m_q\lambda^n_r$, where $\{l,m,n\}\in\{1,2,3,4\}\backslash \{i\}$ and $\{p,q,r\}\in\{5,6,7,8\}\backslash \{j\}$.

Using your favourite equation solving software, we can solve for $\lambda^1_5,\lambda^1_6,\lambda^1_7$ and $\lambda^1_8$, which gives us four very complicated expressions which we will not give here.
To show that Equation \eqref{eqn:tauorig} is satisfied, it remains to show that 
\begin{align*}
 \sum_{\{i,j\}=\{5,7\},\{6,8\}} \epsilon_{ij}(\lambda^1_i\lambda^4_j+\lambda^2_i\lambda^3_j)+\sum_{\{i,j\}=\{6,7\},\{5,8\}} \epsilon_{ij}(\lambda^1_i\lambda^3_j-\lambda^2_i\lambda^4_j)=0, \\
\sum_{\{i,j\}=\{5,6\},\{7,8\}} \epsilon_{ij}(\lambda^1_i\lambda^4_j+\lambda^2_i\lambda^3_j)-\sum_{\{i,j\}=\{5,8\},\{6,7\}} \epsilon_{ij}(\lambda^1_i\lambda^2_j+\lambda^3_i\lambda^4_j)=0, \\
\sum_{\{i,j\}=\{5,6\},\{7,8\}} \epsilon_{ji}(\lambda^2_i\lambda^4_j-\lambda^1_i\lambda^3_j)-\sum_{\{i,j\}=\{5,7\},\{6,8\}} \epsilon_{ij}(\lambda^1_i\lambda^2_j+\lambda^3_i\lambda^4_j)=0,
\end{align*}
where $\epsilon_{ij}=-\epsilon_{ji}$ and $\epsilon_{75}=\epsilon_{68}=\epsilon_{56}=\epsilon_{67}=\epsilon_{78}=\epsilon_{58}=1$. Substituting in the values of $\lambda^1_5,\lambda^1_6,\lambda^1_7$ and $\lambda^1_8$ we found when we solved Equations \eqref{eqn:taunonlin}, these three equations vanish. Therefore, $\tau_{f(y)}(v_1,v_2,v_3,v_4)=0$ if, and only if, $\pi\circ\tau_{f(y)}(v_1,v_2,v_3,v_4)=0$. Since $y\in N$ and $N'$ were arbitrary, it follows that the kernel of $\pi(*_N\exp^*_{\tilde{v}}(\tau|_{N_{\tilde{v}}}))$ and $*_N\exp^*_{\tilde{v}}(\tau|_{N_{\tilde{v}}})$ are the same.
\end{proof}
\begin{rem}
Proposition \ref{prop:caypde} includes a proof of a slightly more general version of a result of Harvey and Lawson \cite[IV.2.C Thm 2.20]{MR666108}.
\end{rem}

\subsection{Properties of the partial differential operator \texorpdfstring{$F$}{F}}
We will now find the linear part of the operator \eqref{eqn:caypde}.

\begin{prop}\label{prop:caylin}
 Let $(M,J,\omega,\Omega)$ be a four-dimensional Calabi--Yau manifold and let $N$ be a two-dimensional compact complex submanifold of $M$. The linearisation of the operator \eqref{eqn:caypde} at zero is given by the elliptic operator
\begin{equation}\label{eqn:infop}
 \infop:C^\infty(\nu^{1,0}_M(N)\oplus \Lambda^{0,2}N\otimes \nu^{1,0}_M(N))\to C^\infty(\Lambda^{0,1}N\otimes \nu^{1,0}_M(N)).
\end{equation}
\end{prop}
\begin{rem}
 We call the forms in the kernel of $\infop$ \emph{infinitesimal Cayley deformations} of $N$ in $M$.
\end{rem}
\begin{proof}
By Proposition \ref{prop:normal}, we can write
\[
v=v_1\oplus \frac{1}{4}(v_2\hook \overline{\Omega})^\sharp, \quad \tilde{v}=v_1\oplus v_2.
\]
To see that \eqref{eqn:infop} is the linearisation of the operator $F$ in Equation \eqref{eqn:caypde}, we make an explicit computation. By definition, we have that
\begin{align*}
 dF|_0({\tilde{v}})=\frac{d}{dt}F(t{\tilde{v}})|_{t=0}=\Psi\left(*_N\mathcal{L}_{\tilde{v}}\tau|_N\right),
\end{align*}
by definition of the Lie derivative. We have that
\[
 *_N\mathcal{L}_{\tilde{v}}\tau|_N=(\mathcal{L}_{\tilde{v}}\tau)(e_1,e_2,e_3,e_4),
\]
where $\{e_1,\dots,e_4\}$ is an orthonormal frame for $TN$ with $\text{vol}_N(e_1,e_2,e_3,e_4)=1$, and so we may apply a formula linking the Lie derivative to the Levi-Civita connection such as \cite[Eqn (4.3.26)]{MR2829653} to find that
\begin{align*}
 (\mathcal{L}_{\tilde{v}}\tau)(e_1,e_2,e_3,e_4)&=(\nabla_{\tilde{v}}\tau)(e_1,e_2,e_3,e_4)+\tau(\nabla_{e_1}{\tilde{v}},e_2,e_3,e_4) \\
&-\tau(\nabla_{e_2}{\tilde{v}},e_1,e_3,e_4)+\tau(\nabla_{e_3}{\tilde{v}},e_1,e_2,e_4) \\
&-\tau(\nabla_{e_4}{\tilde{v}},e_1,e_2,e_3).
\end{align*}
We can write the Levi-Civita connection on $TM|_N$ as $\nabla=\nabla^T+\nabla^\perp$, where $\nabla^T$ is the projection of $\nabla$ onto $T^*N\otimes TN$ and $\nabla^\perp$ is the projection of $\nabla$ onto $T^*N\otimes \nu_M(N)$. Then
\[
 \tau(\nabla^T_{e_i}{\tilde{v}},e_j,e_k,e_l)=0,
\]
for all $\{i,j,k,l\}=\{1,2,3,4\}$ (since $N$ is Cayley), and therefore we have that
\begin{align*}
(\mathcal{L}_{\tilde{v}}\tau)(e_1,e_2,e_3,e_4)&=(\nabla_{\tilde{v}}\tau)(e_1,e_2,e_3,e_4)+\tau(\nabla^\perp_{e_1}{\tilde{v}},e_2,e_3,e_4) \\
&-\tau(\nabla^\perp_{e_2}{\tilde{v}},e_1,e_3,e_4)+\tau(\nabla^\perp_{e_3}{\tilde{v}},e_1,e_2,e_4) \\
&-\tau(\nabla^\perp_{e_4}{\tilde{v}},e_1,e_2,e_3).
\end{align*}
We can calculate that, since $\text{vol}_N=e^1\wedge e^2\wedge e^3\wedge e^4$,
\[
 \begin{matrix}
  \Phi(\,\cdot\, ,e_1, e_2, e_3)=-e^4, & \Phi(\, \cdot\, ,e_1, e_2, e_4)=e^3, \\
\Phi(\, \cdot\, ,e_1, e_3, e_4)=-e^2, & \Phi(\,\cdot\, ,e_2, e_3,e_4)=e^1.
 \end{matrix}
\]
Therefore by definition of $\tau$ (see Proposition \ref{prop:tau}) we have that
\begin{equation}\label{eqn:linmid}
 (\mathcal{L}_{\tilde{v}}\tau)(e_1,e_2,e_3,e_4)=(\nabla_{\tilde{v}}\tau)(e_1,e_2,e_3,e_4)+\sum_{i=1}^4\pi_7(e^i\wedge(\nabla^\perp_{e_i}{\tilde{v}})^\flat).
\end{equation}
It remains to show that since $\Phi$ is parallel, $\tau$ is parallel.
Extending $e_1,\dots, e_4$ to an orthonormal frame $e_1,\dots,e_8$ for $TM|_N$ and using Equation \eqref{eqn:taunice}
\begin{align*}
\nabla_{\tilde{v}}\tau&=\sum_{i=2}^8\nabla_{\tilde{v}}\left[e^i\wedge (e_1\hook \Phi)-e^1\wedge (e_i\hook\Phi)\right]\otimes \pi_7(e^1\wedge e^i) \\
&+\sum_{i=2}^8\left[e^i\wedge (e_1\hook \Phi)-e^1\wedge (e_i\hook\Phi)\right]\otimes \nabla_{\tilde{v}}\pi_7(e^1\wedge e^i).
\end{align*}
We can see that the second sum in the above expression will vanish when evaluated on $e_1,e_2,e_3,e_4$, so it remains to compute
\[
\nabla_{\tilde{v}}\left[e^i\wedge (e_1\hook \Phi)-e^1\wedge (e_i\hook\Phi)\right],
\]
for $i=2,\dots ,8$. Since 
\[
\nabla_{\tilde{v}}(e^1\wedge (e_i\hook \Phi))=(\nabla_{\tilde{v}} e^1)\wedge (e_i\hook \Phi)+e^1\wedge (\nabla_{\tilde{v}}e_i\hook \Phi)+e^1\wedge (e_i\hook\nabla_{\tilde{v}}\Phi),
\]
we find that
\begin{align*}
\nabla_{\tilde{v}}(e^1\wedge (e_i\hook \Phi))(e_1,e_2,e_3,e_4)&=e^2(e_i)(\nabla_{\tilde{v}}e^1)(e_2)+e^3(e_i)(\nabla_{\tilde{v}} e^1)(e_3) \\
&+e^4(e_i)(\nabla_{\tilde{v}}e^1)(e_4)+\Phi(\nabla_{\tilde{v}} e_i,e_2,e_3,e_4) \\
&+(\nabla_{\tilde{v}}\Phi)(e_i,e_2,e_3,e_4).
\end{align*}
Similarly,
\begin{align*}
\nabla_{\tilde{v}}(e^i\wedge (e_1\hook \Phi))(e_1,e_2,e_3,e_4)&=(\nabla_{\tilde{v}} e^i)(e_1)-e^i(e_2)\Phi(\nabla_{\tilde{v}}e_1,e_1,e_3,e_4) \\
&+e^i(e_3)\Phi(\nabla_{\tilde{v}}e_1,e_1,e_2,e_4) \\
&-e^i(e_4)\Phi(\nabla_{\tilde{v}}e_1,e_1,e_2,e_3).
\end{align*}
Using the explicit expression for $\Phi$, we have that
\begin{align*}
(\nabla_{\tilde{v}}\tau)(e_1,e_2,e_3,e_4)&=\sum_{i=2}^8\left[e^i(e_2)e^2(\nabla_{\tilde{v}}e_1)-e^2(e_i)(\nabla_{\tilde{v}}e^1)(e_2) \right. \\
+e^i(e_3)e^3(\nabla_{\tilde{v}}e_1)&-e^3(e_i)(\nabla_{\tilde{v}}e^1)(e_3)+e^i(e_4)e^4(\nabla_{\tilde{v}}e_1)-e^4(e_i)(\nabla_{\tilde{v}}e^1)(e_4) \\
-e^1(\nabla_{\tilde{v}} e_i)&\left.+(\nabla_{\tilde{v}} e^i)(e_1)-(\nabla_{\tilde{v}}\Phi)(e_i,e_2,e_3,e_4)\right]\otimes \pi_7(e^1\wedge e^i).
\end{align*}
Finally, note that since the metric $g$ on $X$ is parallel with respect to the Levi-Civita connection,
\begin{align*}
(\nabla_{\tilde{v}}e^j)(e_k)&=-e^j(\nabla_{\tilde{v}} e_k)=-g(\nabla_{\tilde{v}} e_k,e_j)=g(e_k,\nabla_{\tilde{v}} e_j)=e^k(\nabla_{\tilde{v}} e_j) \\
&=-(\nabla_{\tilde{v}} e^k)(e_j),
\end{align*}
and so we find that
\[
(\nabla_{\tilde{v}}\tau)(e_1,e_2,e_3,e_4)=\sum_{i=5}^8-(\nabla_{\tilde{v}}\Phi)(e_i,e_1,e_2,e_3,e_4)\otimes\pi_7(e^1\wedge e^i),
\]
which vanishes since $\Phi$ is parallel.

It remains to show that 
\[
\Psi\left(\sum_{i=1}^4\pi_7(e^i\wedge(\nabla^\perp_{e_i}(v_1+v_2))^\flat)\right)=\bar\partial v_1+\bar\partial^*\frac{1}{4}(v_2\hook \overline{\Omega})^\sharp.
\]
We have that
\[
\sum_{i=1}^4 e^i\wedge \nabla_{e_i}^\perp v_1=:\partial v_1+\bar\partial v_1.
\]
By Proposition \ref{prop:e}, we have that
\[
\Psi\circ\pi_7((\partial v_1)^\flat+(\bar\partial v_1)^\flat)=\Psi\circ\pi_7 ((\bar\partial v_1)^\flat)=\bar\partial v_1,
\]
where $\flat:\nu^{1,0}_M(N)\to \nu^{*0,1}_M(N)$. Now since $v_2$ is of type $(0,1)$,
\begin{align*}
\sum_{i=1}^4\pi_7(e^i\wedge(\nabla^\perp_{e_i}v_2)^\flat)&=\sum_{i=1}^4\frac{1}{2}\left[e^i\wedge(\nabla^\perp_{e_i}v_2)^\flat+\frac{1}{2}\omega\wedge\omega(e_i,\nabla_{e_i}^\perp v_2,\,\cdot\, ,\, \cdot\, )\right. \\
&\left.+\frac{1}{2}\overline{\Omega}(e_i,\nabla_{e_i}^\perp v_2,\,\cdot\, ,\, \cdot\, )\right] \\
&=\sum_{i=1}^4 \frac{1}{2}g(e_i+iJe_i,\, \cdot\, )\wedge (\nabla_{e_i}^\perp v_2)^\flat+\frac{1}{4}\overline{\Omega}(e_i,\nabla_{e_i}^\perp v_2,\,\cdot\, ,\, \cdot\, ),
\end{align*}
and so under an application of $\Psi$, we have that
\[
\Psi\left(\sum_{i=1}^4\pi_7(e^i\wedge(\nabla^\perp_{e_i}v_2)^\flat)\right)=\frac{1}{4}\overline{\Omega}(e_i,\nabla_{e_i}^\perp v_2,\,\cdot\, ,\, \cdot\, )^\sharp,
\]
where $\sharp:\nu^{*0,1}_M(N)\to \nu^{1,0}_M(N)$ denotes the musical isomorphism. Finally, since $\Omega$ is parallel and we are using the Levi-Civita connection, we have that
\[
\overline{\Omega}(e_i,\nabla_{e_i}^\perp v_2,\,\cdot\, ,\, \cdot\, )^\sharp=-\sum_{i=1}^4e_i\hook \nabla_{e_i}(v_2\hook \overline{\Omega})^\sharp=:\bar\partial^*(v_2\hook\overline{\Omega})^\sharp,
\]
and so we are done.
\end{proof}

\subsection{The moduli space of Cayley deformations}\label{ss:modspcaydef}

We will now prove that we can extend the operator \eqref{eqn:caypde} to a smooth map of Banach spaces. The argument we use to prove Lemma \ref{lem:cayreg} is reasonably standard, and is based on the arguments in \cite[Prop 2.10]{MR2054572} and \cite[Prop 6.9]{MR2403190}.

\begin{lem}\label{lem:cayreg}
 Let $(M,J,\omega,\Omega)$ be a four-dimensional Calabi--Yau manifold and let $N$ be a two-dimensional compact complex submanifold of $M$. Let $F$ be the partial differential operator defined in Equation \eqref{eqn:caypde}. Then we can extend $F$ to a smooth map of Banach spaces
\begin{equation}\label{eqn:lpcaypde}
 F:L^p_{k+1}(U)\to L^p_{k}(\Lambda^{0,1}N\otimes \nu^{1,0}_M(N)),
\end{equation}
for any $1<p<\infty$ and $k\in\mathbb{N}$ satisfying $k>1+4/p$. Moreover, the normal vector fields in the kernel of \eqref{eqn:lpcaypde} are smooth.
\end{lem}
\begin{proof}
At each point $y$ of $N$ we have that $F(v)(y)$ relates to the tangent space of the deformation $N_{\tilde{v}}:=\exp_{\tilde{v}}(N)$ and therefore depends on $v$ and $\nabla v$. We may write
\begin{equation}\label{eqn:defineQ}
 F(v)(x)=(\infop)v(x)+Q(x,v(x),\nabla v(x)),
\end{equation}
by Proposition \ref{prop:caylin}, and use Equation \eqref{eqn:defineQ} to define $Q$ to be a map
\[
 \{(x,y,z)\, |\, (x,y)\in U, z\in T_x^*N\otimes \nu_{x}(N)\}\to \Lambda^{0,1}N\otimes \nu^{1,0}_M(N),
\]
so that $Q(v)(x):=Q(x,v(x),\nabla v(x))$ is a section of $\Lambda^{0,1}N\otimes \nu^{1,0}_M(N)$. By definition of $F$, $Q$ is smooth in $x,y$ and $z$.
Since we can think of $Q$ as a map $\nu_x(N)\otimes T^*_xN\otimes \nu_x(N)\to [\Lambda^{0,1}N\otimes \nu^{1,0}_M(N)]_x$, we can make sense of a Taylor expansion of $Q(x,y,z)$ around $(x,0,0)$. Since by definition $Q$ has no linear part at zero we deduce that
\[
 |Q(x,y,z)| \le C_x(|y|+|z|)^2,
\]
for each $x\in N$. Since $N$ is compact, we may deduce that
\[
 \|Q(v)(x)\|_{C^0} \le C\|v\|_{C^1}^2,
\]
where $C$ is independent of $x$.
From this we see that
\begin{align*}
 \left(\int_N |Q(v)(x)|^p \text{ vol}_N\right)^{1/p} &\le C\left(\int_N(|v|+|\nabla v|)^{2p}\text{ vol}_N\right)^{1/p} \\
&\le C \|v\|_{C^1}\left(\int_N(|v|+|\nabla v|)^p\text{ vol}_N\right)^{1/p} \\
&\le C \|v\|_{C^1}\left(\int_N |v|^p\text{ vol}_N+\int_N |\nabla v|^p \text{ vol}_N\right)^{1/p},
\end{align*}
by Minkowski's inequality. So we have $Q$ maps $L^p_1(\nu_M(N))\cap C^1(\nu_M(N))\to L^p(\Lambda^{0,1}N\otimes \nu^{1,0}_M(N))$. We can take the derivative of the Taylor expansion of $Q$, and apply the chain rule to estimate $|\nabla Q|$ by a polynomial in $|v|, |\nabla v|$ and $|\nabla^2 v|$. A similar argument to the $k=0$ case given above shows that for each $k\in\mathbb{N}$ there exists $C_k>0$ so that 
\begin{equation}\label{eqn:qestderiv}
 \|Q(v)\|_{p,k}\le C_k\|v\|_{C^1}\|v\|_{p,k+1}.
\end{equation}
In particular, when $k>4/p$, $L^p_{k+1}(\nu_M(N))$ is continuously embedded in $C^1(\nu_M(N))$ by \cite[Thm 2.10]{MR1636569}, and so for $k>4/p$ there exist $\tilde{C}_k>0$ so that
\begin{equation}\label{eqn:qestderiv2}
 \|Q(v)\|_{p,k}\le \tilde{C}_k\|v\|_{p,k+1}^2.
\end{equation}
Since $\infop$ is linear, we see that $F$ takes $L^p_{k+1}(\nu_M(N))$ into $L^p_k(\Lambda^{0,1}N\otimes \nu^{1,0}_M(N))$.

Now we must show that \eqref{eqn:lpcaypde} is a smooth map of Banach spaces. Firstly, since 
\[
 v\mapsto (\infop)v,
\]
is linear, it is clearly smooth as a map
\[
 L^p_{k+1}(U)\to L^p_k(\Lambda^{0,1}N\otimes \nu^{1,0}_M(N)).
\]
To see that
\[
 v \mapsto (x\mapsto Q(x,v(x),\nabla v(x))),
\]
is a smooth map
\[
 L^p_{k+1}(U)\to L^p_k(\Lambda^{0,1}N\otimes \nu^{1,0}_M(N)),
\]
we proceed as follows. To see that $F$ is once differentiable at zero in this sense, notice that
\[
 \frac{\|F(v)-F(0)-(\infop)v\|_{p,k}}{\|v\|_{p,k+1}}=\frac{\|Q(v)\|_{p,k}}{\|v\|_{p,k+1}}\to 0,
\]
as $\|v\|_{p,k+1}\to 0$ by the estimate \eqref{eqn:qestderiv2}. Repeating this argument for the derivatives of $Q$, we can show that we can differentiate $Q$ as many times as we like. We deduce that \eqref{eqn:lpcaypde} is a smooth map of Banach spaces.

Finally, regularity of the kernel of \eqref{eqn:lpcaypde} follows from a nonlinear elliptic regularity result, such as \cite[Thm 3.56]{MR1636569}, which we may apply since $k>1+4/p$ (which allows us to embed $L^p_{k+1}(U)$ in $C^2(U)$ by Sobolev embedding \cite[Thm 2.10]{MR1636569}).
\end{proof}

We will now deduce the main result of this section. For the reader's convenience, we will present the Banach space implicit function theorem here in the form that we will need it. See, for example, \cite[Ch 6 Thm 2.1]{MR783635} for a proof.
 
\begin{thm}[Implicit function theorem]\label{thm:imp}
 Let $X$ and $Y$ be Banach spaces and let $U\subseteq X$ be an open neighbourhood of zero. Let $\mathcal{F}:U\to Y$ be a $C^k$-map, with $k\ge 1$, such that $\mathcal{F}(0)=0$. Suppose further that $d\mathcal{F}\vert_0:X\to Y$ is surjective, with kernel $K$ such that $X=K\oplus X'$ for some closed subspace $X'$ of $X$.

Then there exist open sets $K_0\subseteq K$, $X'_0\subseteq X'$ both containing zero and a $C^k$-map $g:K_0\to X'_0$ such that $g(0)=0$ and
\[
 \mathcal{F}^{-1}(0)\cap( K_0\times X'_0)=\{(x,g(x))\,|\, x\in K_0\}.
\]
\end{thm}

\begin{thm}\label{thm:compactcay}
Let $(M,J,\omega,\Omega)$ be a four-dimensional Calabi--Yau manifold and let $N$ be a two-dimensional compact complex submanifold of $M$. Then there exist a smooth manifold $K_0$, which is an open neighbourhood of $0$ in $\textnormal{Ker }(\infop)$, and a smooth map $g_2:K_0 \to \textnormal{Ker }(\infop)^*$ with $g(0)=0$ so that an open neighbourhood of $N$ in the moduli space of Cayley deformations of $N$ in $M$ is homeomorphic to an open neighbourhood of $0$ in $\textnormal{Ker }g_2$.

Moreover, the expected dimension of the moduli space of Cayley deformations of $N$ in $M$ is given by
\[
 \textnormal{ind } (\infop):= \textnormal{dim }\textnormal{Ker }(\infop)-\textnormal{dim }\textnormal{Ker } (\infop)^*,
\]
where
\begin{align*}
 (\infop)^*:C^\infty(\Lambda^{0,1}N\otimes \nu^{1,0}_M(N))\to C^\infty(\nu^{1,0}_M(N)\oplus \Lambda^{0,2}N\otimes \nu^{1,0}_M(N)),
\end{align*}
is the formal adjoint of $\infop$. If $\textnormal{Ker } (\infop)^*=\{0\}$ then the moduli space of Cayley deformations of $N$ in $M$ is a smooth manifold near $N$ of dimension
\[
\textnormal{dim Ker }(\infop).
\]
\end{thm}
\begin{proof}
 By Proposition \ref{prop:caypde} we know that the moduli space of Cayley deformations of $N$ in $M$ is locally homeomorphic to the kernel of $F$ given in \eqref{eqn:caypde}.
By Lemma \ref{lem:cayreg}, without changing the kernel, $F$ extends to a smooth map
\[
 L^p_{k+1}(U)\to L^p_{k}(\Lambda^{0,1}N\otimes \nu^{1,0}_M(N)),
\]
for any $1<p<\infty$ and $k\in \mathbb{N}$ and the linearisation of $F$ at zero is the elliptic operator $\infop$, which extends by density to a smooth map
\begin{equation}\label{eqn:dmainthm}
\infop:L^p_{k+1}(\nu^{1,0}_M(N)\oplus \Lambda^{0,2}N\otimes \nu^{1,0}_M(N))\to L^p_{k}(\Lambda^{0,1}N\otimes \nu^{1,0}_M(N)).
\end{equation}
Since $N$ is compact and \eqref{eqn:dmainthm} is elliptic, the map \eqref{eqn:dmainthm} is Fredholm, and therefore \eqref{eqn:dmainthm} has finite-dimensional kernel and cokernel, and closed image. As a consequence, we can write
\[
L^p_{k+1}(\nu^{1,0}_M(N)\oplus \Lambda^{0,2}N\otimes \nu^{1,0}_M(N))=K'\oplus X', 
\]
where $K'$ is the kernel of $\infop$ and $X'$ is closed, and 
\[
L^p_{k}(\Lambda^{0,1}N\otimes \nu^{1,0}_M(N))=(\infop)L^p_{k+1}(\nu^{1,0}_M(N)\oplus \Lambda^{0,2}N\otimes \nu^{1,0}_M(N))\oplus \mathcal{O},
\]
where $\mathcal{O}$ is a finite-dimensional space that we'll call the \emph{obstruction space}, and
\begin{align*}
 \mathcal{O}&\cong L^p_{k}(\Lambda^{0,1}N\otimes \nu^{1,0}_M(N))/(\infop)L^p_{k+1}(\nu^{1,0}_M(N)\oplus \Lambda^{0,2}N\otimes \nu^{1,0}_M(N)) \\
 &=: \text{Coker } (\infop).
\end{align*}
Notice that if the obstruction space vanishes, i.e., $\mathcal{O}=\{0\}$, then it follows immediately from the implicit function theorem \ref{thm:imp} that the moduli space of Cayley deformations of $N$ is a smooth manifold near $N$ of dimension $\text{dim Ker }(\infop)$. However, the obstruction space is nontrivial in general, and so $\infop$ is not surjective, thus we are not able to apply the implicit function theorem \ref{thm:imp} to $F$. Instead define
\begin{align*}
 \mathcal{F}:L^p_{k+1}(U)\times \mathcal{O}&\to L^p_{k}(\Lambda^{0,1}N\otimes \nu^{1,0}_M(N)), \\
(v,w)&\mapsto F(v)+w. 
\end{align*}
We see that
\[
 d\mathcal{F}|_{(0,0)}(v,w)=(\infop)v+w,
\]
which surjects, and therefore we may apply the implicit function theorem \ref{thm:imp} to $\mathcal{F}$. Denoting the kernel of $d\mathcal{F}|_{(0,0)}$ by $K=K'\times\{0\}$, we can write
\[
 L^p_{k+1}(U)\times \mathcal{O}=K\oplus (X'\times \mathcal{O}).
\]
The implicit function theorem \ref{thm:imp} gives us open sets $K_0\subseteq K$, $X'_0\subseteq X'$ and $\mathcal{O}_0\subseteq\mathcal{O}$ and a smooth map $g=(g_1,g_2):K_0\to X'_0\times\mathcal{O}_0$ such that
\[
 \mathcal{F}^{-1}(0)\cap (K_0\times X'_0\times \mathcal{O}_0)=\{(x,g_1(x),g_2(x))\,|\,x\in K_0\}.
\]
Then for $x\in K_0$ we have that
\[
 \mathcal{F}(x,g_1(x),g_2(x))=F(x,g_1(x))+g_2(x)=0.
\]
Therefore we can identify the kernel of $F$ with the kernel of the map $g_2:K_0\to \mathcal{O}_0$. These spaces are finite-dimensional since $\infop$ is Fredholm. By Sard's theorem, we may deduce that the expected dimension of the kernel of $g_2$ is equal to the difference of the dimensions of $K_0$ and $\mathcal{O}_0$, and therefore the expected dimension of the moduli space of Cayley deformations of $N$ in $M$ is
\[
 \text{dim Ker }(\infop)-\text{dim Coker }(\infop),
\]
where $\infop$ is considered as a map $L^p_{k+1}(U)\to L^p_k(\Lambda^{0,1}N\otimes \nu^{1,0}_M(N))$. We have that the cokernel $\infop$ is isomorphic to the kernel of the adjoint to $\infop$, $(\infop)^*$, since $N$ is compact. Elliptic regularity tells us that the kernels of $\infop$ and $(\infop)^*$ acting on $L^p_{k+1}(\nu^{1,0}_M(N)\oplus \Lambda^{0,2}N\otimes \nu^{1,0}_M(N))$ and $(L^p_{k}(\Lambda^{0,1}N\otimes \nu^{1,0}_M(N)))^*$ for any $1<p<\infty$ and $k\in\mathbb{N}$ are exactly equal to the kernels of $\infop$ and $(\infop)^*$ acting on $C^\infty(\nu^{1,0}_M(N)\oplus \Lambda^{0,2}N\otimes \nu^{1,0}_M(N))$ and $C^\infty(\Lambda^{0,1}N\otimes \nu^{1,0}_M(N))$ respectively.
\end{proof}

\subsection{Index theory}\label{sec:ind}

We will now compute the expected dimension of the moduli space from Theorem \ref{thm:compactcay} in terms of topological invariants of the manifold.

\begin{thm}\label{thm:index}
 Let $N$ be a two-dimensional compact complex submanifold of a four-dimensional Calabi--Yau manifold $M$. Consider the operator
 \[
 \infop:C^\infty(\nu_M^{1,0}(N)\oplus \Lambda^{0,2}M\otimes \nu^{1,0}_M(N))\to C^\infty(\Lambda^{0,1}N\otimes \nu^{1,0}_M(N)).
\]
Then the index of this operator is given by
\begin{equation}\label{eqn:indformula}
 \textnormal{ind } \infop=\frac{1}{2}\textnormal{sign}(N)+\frac{1}{2}\chi(N)-[N]\cdot[N],
\end{equation}
where $\textnormal{sign}(N)$ is the signature of $N$, $\chi(N)$ is the Euler characteristic of $N$ and $[N]\cdot[N]$ is the self-intersection number of $N$.
\end{thm}
\begin{proof}
 Since $N$ is compact, we can identify the kernel of $\infop$ and the kernel of its adjoint with Dolbeault cohomology groups. That is,
\begin{align*}
 \text{dim}_\mathbb{C}\text{Ker }(\infop)&=\text{dim}_\mathbb{C} H^{0,0}_{\bar\partial}(N,\nu^{1,0}_M(N))+\text{dim}_\mathbb{C}H^{0,2}_{\bar\partial}(N,\nu^{1,0}_M(N)), \\
\text{dim}_\mathbb{C} \text{Ker }(\infop)^*&=\text{dim}_\mathbb{C}H^{0,1}_{\bar\partial}(N,\nu^{1,0}_M(N)).
\end{align*}
By Dolbeault's theorem, we can then identify the index of the operator with the dimensions of certain sheaf cohomology groups. We have that
\[
 \text{ind }\infop=\sum_{i=0}^2(-1)^i\text{dim}_\mathbb{C}H^i(N,\nu^{1,0}_M(N)).
\]
Then by the Hirzebruch--Riemann--Roch theorem \cite[Thm 5.1.1]{MR2093043}, we have that
\[
 \text{ind }\infop=\int_N \text{ch}(\nu^{1,0}_M(N))\text{td}(N),
\]
where ch$(\nu^{1,0}_M(N))$ is the Chern character of $\nu^{1,0}_M(N)$ and td$(N)$ is the Todd class of $N$.

We calculate that
\begin{align*}
 \int_N \text{ch}(\nu^{1,0}_M(N))\text{td}(N)&=\int_N\frac{1}{6}(c_1^2(N)+c_2(N))+\frac{1}{2}c_1(\nu^{1,0}_M(N))c_1(N) \\
&+\frac{1}{2}(c_1^2(\nu^{1,0}_M(N))-2c_2(\nu^{1,0}_M(N))).
\end{align*}
Since $M$ is a Calabi--Yau manifold, $c_1(M)=0$, and therefore
\[
 0=c_1(T^{1,0}M|_N)=c_1(T^{1,0}N\oplus \nu^{1,0}_M(N))=c_1(T^{1,0}N)+c_1(\nu^{1,0}_M(N)),
\]
which tells us that 
\begin{align*}
 \int_N \text{ch}(\nu^{1,0}_M(N))\text{td}(N)&=\frac{1}{6}(c_1^2(N)+c_2(N))-c_2(\nu^{1,0}_M(N)) \\
&=\frac{1}{6}(c_1^2(N)-2c_2(N))+\frac{1}{2}c_2(N)-c_2(\nu^{1,0}_M(N)).
\end{align*}
Finally, since $c_i(\overline{E})=(-1)^ic_i(E)$,
\begin{align*}
 c_2(TN\otimes \mathbb{C})&=c_2(T^{1,0}N\oplus T^{0,1}N)=c_2(N)+c_1(N)c_1(T^{0,1}N)+c_2(T^{0,1}N) \\
&=2c_2(N)-c_1(N)^2,
\end{align*}
and so by definition of the Pontryagin class $p_1(N)$, we see that
\[
  \int_N \text{ch}(\nu^{1,0}_M(N))\text{td}(N)=\frac{1}{6}p_1(N)+\frac{1}{2}c_2(N)-c_2(\nu^{1,0}_M(N)),
\]
and therefore applying the Hirzebruch signature theorem \cite[Cor 5.1.4]{MR2093043} we have that
\[
 \text{ind }\infop=\frac{1}{2}\text{sign}(N)+\frac{1}{2}\chi(N)-[N]\cdot[N],
\]
as required.
\end{proof}

\section{Complex deformations of compact complex submanifolds}\label{sec:compdef}

Our ultimate goal in this section is to find out when a Cayley deformation of a compact complex surface $N$ in a Calabi--Yau four-fold $M$ is a complex deformation. We will deduce this as a corollary of a result on the complex deformations of any compact complex submanifold of a Calabi--Yau manifold.

If $N'$ is a Cayley deformation of $N$, we see that
\[
 \text{vol}_{N'}=\frac{1}{2}\omega\wedge \omega|_{N'}+\text{Re }\Omega|_{N'},
\]
where $\omega$ is the Ricci-flat K\"ahler form and $\Omega$ is the holomorphic volume form of $M$. It is easy to see that $N'$ is a complex submanifold of $M$ if, and only if,
\[
 \text{Re }\Omega|_{N'}\equiv 0.
\]
It turns out that we can use the holomorphic volume form to define a form that vanishes exactly when restricted to any complex submanifold of a given dimension.

\subsection{A form that vanishes on complex submanifolds}\label{ss:comp}
In this section will will prove that there exists a differential form on a Calabi--Yau manifold that vanishes if and only if restricted to a complex submanifold of a given dimension.
We will first require a result which follows from a lemma of Harvey and Lawson \cite[II.6 Lem 6.13]{MR666108}.
\begin{lem}\label{lem:specialform}
 Let $V$ be a $2p$-dimensional oriented linear subspace of $\mathbb{C}^m$, with $2p\le m$. Then there exist a unitary basis $e_1,Je_1,\dots,e_m,Je_m$ for $\mathbb{C}^m$ and angles $0\le \theta_1\le \theta_2\le \dots \le \theta_{p-1}\le\pi/2$, $\theta_{p-1}\le \theta_p\le \pi$ such that
\begin{align*}
 V=\textnormal{span}\{e_1,Je_1\cos\theta_1+e_2\sin\theta_1,e_3,\dots , e_{2p-1},Je_{2p-1}\cos\theta_p+e_{2p}\sin\theta_p\}.
\end{align*}
If $2p>m$, then we have that
\begin{align*}
V&=\textnormal{span}\{e_1,Je_1\cos\theta_1+e_2\sin\theta_1,\dots , e_{2(n-p)-1}, \\ 
Je_{2(n-p)-1}\cos\theta_{n-p}&+e_{2(n-p)}\sin\theta_{n-p},e_{2(n-p)+1},Je_{2(n-p)+1},\dots , e_n,Je_n\}.
\end{align*}
\end{lem}
Given this result, we can prove the following proposition.
\begin{prop}\label{prop:compdefs}
 Let $X$ be an oriented real $2p$-dimensional submanifold of an $m$-dimensional Calabi--Yau manifold $M$, with $m\ne p+1$. Then $X$ is a complex submanifold of $M$ if, and only if
\[
 \sigma(v_1,\dots , v_{p+1})\equiv0,
\]
for all vector fields $v_1,\dots v_{p+1}$ on $X$, where 
\begin{equation}\label{eqn:sigma}
 \sigma(v_1,\dots , v_{p+1}):=\textnormal{Re }\Omega(v_1,\dots , v_{p+1},\,\cdot\, ,\dots, \, \cdot \,),
\end{equation}
where $\Omega$ is the holomorphic volume form of $M$.

If $m=p+1$ then we must have that
\[
\textnormal{Re }\Omega(v_1,\dots , v_{p+1})=0=\textnormal{Im }\Omega(v_1,\dots , v_{p+1}),
\]
for all vector fields $v_1,\dots v_{p+1}$ on $X$.
\end{prop}
\begin{proof}
If $X$ is complex, then by the adjunction formula \cite[Prop 2.2.17]{MR2093043},
\[
 K_M|_X\cong K_X\otimes \Lambda^{m-p}\nu^{*1,0}_M(X),
\]
where $K_M$ denotes the canonical bundle of $M$. Since $\Omega$ is a nowhere vanishing section of $K_M$, it is easy to see that for any $p+1$ vector fields $v_1,\dots,v_{p+1}$ on $X$
\[
 \Omega(v_1,\dots ,v_{p+1},\,\cdot\, , \dots ,\,\cdot \,)=\overline{\Omega}(v_1,\dots ,v_{p+1},\,\cdot\, , \dots ,\,\cdot \,)=0,
\]
and so,
\[
 \text{Re }\Omega(v_1,\dots ,v_{p+1},\,\cdot\, , \dots ,\,\cdot \,)=0.
\]

It remains to show that $\sigma|_X\equiv 0$ implies that $X$ is a complex manifold. We show the contrapositive, that is, if $X$ is not complex, then we can find vector fields $v_1,\dots , v_{p+1}$ on $X$ so that $\sigma(v_1,\dots , v_{p+1})\ne 0$. It suffices to show that for an arbitrary $x\in X$, we can find nonzero $v_1,\dots ,v_{p+1}\in T_xX$ so that $\sigma_x(v_1, \dots , v_{p+1})\ne 0$.

First assume that $2p\le m$. Identifying $(T_xM,\omega_x)$ with $\mathbb{C}^m$ with the standard Euclidean K\"ahler form, we can view $T_xX$ as an oriented $2p$-dimensional linear subspace $V$ of $\mathbb{C}^m$. Apply Lemma \ref{lem:specialform} to choose a unitary basis $\{e_1,Je_1,\dots, e_m,Je_m\}$ for $\mathbb{C}^m$ so that for some $0\le\theta_1\le \dots \le \theta_{p-1}\le \pi/2$, $\theta_{p-1}\le \theta_p\le \pi$,
\begin{align*}
 V=\textnormal{span}\{e_1,Je_1\cos\theta_1+e_2\sin\theta_1,e_3, \dots , e_{2p-1},Je_{2p-1}\cos\theta_p+e_{2p}\sin\theta_p\}.
\end{align*}
Since $V$ is not a complex subspace of $\mathbb{C}^m$, let $j_0\in \{1,\dots , p\}$ be so that that $0<\theta_{j_0}<\pi$, $\theta_j=0$ for $j=1,\dots , j_0-1$. The holomorphic volume form on $\mathbb{C}^m$ takes the form
\[
 \Omega_0=(e^1-iJe^1)\wedge \dots \wedge (e^m-iJe^m), 
\]
where $e^i=g(e_i,\cdot)$. Notice that $Je^i=-g(Je_i,\cdot)$. The holomorphic volume form on $T_xM$ in this choice of basis will take the form $\Omega=e^{i\phi}\Omega_0$ for some $\phi\in [0,2\pi)$. Take 
\[
v_1=e_1,v_2=e_3, \dots , v_p=e_{2p-1}, v_{p+1}=Je_{2j_0-1}\cos\theta_{j_0}+e_{2j_0}\sin\theta_{j_0}.
\]
Then we have that
\begin{align*}
 \sigma(v_1,\dots , v_{p+1})&=\sin\theta_{j_0} \text{Re }(e^{i\phi}\Omega_0)(e_1,\dots ,e_{2p-1},e_{2j_0},\,\cdot \, , \dots , \,\cdot \,),
\end{align*}
which doesn't vanish regardless of the value of $\phi$ since we assumed that $0<\theta_{j_0}<\pi$ -- as long as $p+1\ne m$, which can only happen when $p=1$, $m=2$. In this case, we find that 
\[
\sigma(e_1,Je_1\cos\theta_1+e_2\sin\theta_1)=\sin\theta_1 \cos\phi,
\]
which will vanish if $\phi=\pi/2$. But in this case,
\[
\text{Im }\Omega(e_1,Je_1\cos\theta_1+e_2\sin\theta_1)=\sin\theta_1\sin\phi\ne 0,
\]
and so regardless of the value of $\phi$, the proposition holds.

The case $2p>m$ follows from a similar argument.
\end{proof}

In the style of Proposition \ref{prop:caypde} we can now identify the moduli space of complex deformations of a compact complex submanifold in a Calabi--Yau manifold with the kernel of a partial differential operator.
\begin{prop}\label{prop:comppde}
 Let $N^p$ be a compact complex submanifold of a Calabi--Yau manifold $M^m$. Let $V$ be the open set from the tubular neighbourhood theorem \ref{thm:cpttubnbhd}, and for $v\in C^\infty(V)$ define $N_v:=\exp_v(N)$. If $p+1\ne m$, then the moduli space of complex deformations of $N$ in $M$ is locally homeomorphic to the kernel of
\begin{align}\nonumber
 G:C^\infty(V\otimes \mathbb{C})&\to C^\infty(\Lambda^{p-1}N\otimes \Lambda^{m-p-1}T^*M|_N\otimes \mathbb{C}), \\ \label{eqn:cxop}
v&\mapsto *_N\exp_v^*(\sigma|_{N_v}),
\end{align}
where $\sigma$ was defined in Proposition \ref{prop:compdefs}.
If $p+1=m$, then the moduli space of complex deformations of $N$ in $M$ is locally homeomorphic to the intersection of the kernels of 
\begin{align}\nonumber
 G_1:C^\infty(V\otimes \mathbb{C})&\to C^\infty(\Lambda^{p-1}N), \\ \label{eqn:cxop1}
v&\mapsto *_N\exp_v^*(\textnormal{Re }\Omega|_{N_v}),
\end{align}
and
\begin{align}\nonumber
 G_2:C^\infty(V\otimes \mathbb{C})&\to C^\infty(\Lambda^{p-1}N), \\ \label{eqn:cxop2}
v&\mapsto *_N\exp_v^*(\textnormal{Im }\Omega|_{N_v}),
\end{align}
\end{prop}
\begin{proof}
 The identification of the moduli space with the kernel of these operators follows immediately from Proposition \ref{prop:compdefs}.
\end{proof}

\subsection{Properties of the operators \texorpdfstring{$G$, $G_1$ and $G_2$}{G}}\label{ss:compop}

We will now study the operators $G,G_1$ and $G_2$.
\begin{prop}\label{prop:glin}
 Let $N^p$ be a compact complex submanifold of a Calabi--Yau manifold $M^m$. Let $G,G_1$ and $G_2$ be the partial differential operators defined in Equations \eqref{eqn:cxop}, \eqref{eqn:cxop1} and \eqref{eqn:cxop2} respectively. Then the linearisation of these operators at zero is given by
\[
 v\mapsto i^{p+j}(-1)^{\frac{p(p-1)}{2}+1+j}(\partial^*(v\hook\Omega)+ (-1)^{p+j}\bar\partial^*(v\hook \overline{\Omega})),
\]
where $v\in C^\infty(\nu_M(N)\otimes \mathbb{C})$ and we take $j=0$ for $G$ and $G_1$ and $j=1$ for $G_2$. Therefore $v$ is an infinitesimal complex deformation of $N$ if, and only if,
\[
 \partial^*(v\hook \Omega)=0=\bar\partial^*(v\hook \overline{\Omega})
\]

Moreover, we have that, if $v=v_1\oplus v_2$ where $v_1\in \nu^{1,0}_M(N)$ and $v_2\in \nu^{0,1}_M(N)$
\[
 \partial^*(v_1\hook\Omega)=0 \iff \bar\partial v_1=0.
\]
\end{prop}
\begin{cor}
Let $N$ be a two-dimensional compact complex submanifold of a four-dimensional Calabi--Yau manifold $M$. Then infinitesimal complex and Cayley deformations of $N$ are the same.
\end{cor}
\begin{proof}[Proof of Corollary.]
By Propositions \ref{prop:normal} and \ref{prop:caylin}, a complexified normal vector field $v=v_1\oplus v_2$ is an infinitesimal Cayley deformation of $N$ if
\[
\bar\partial v_1+\frac{1}{4}\bar\partial^*(v_2\hook \overline{\Omega})=0.
\]
By Proposition \ref{prop:glin}, $v_1\oplus v_2$ is an infinitesimal complex deformation of $N$ if
\[
\bar\partial v_1=0=\bar\partial^*(v_2\hook \overline\Omega).
\]
The result follows since $N$ is compact.
\end{proof}
\begin{proof}[Proof of Proposition \ref{prop:caylin}.]
 By definition, we have that
\[
 dG|_0(V)=\frac{d}{dt}G(tv)|_{t=0}=*_N(\mathcal{L}_v\sigma|_N).
\]
Choose an orthonormal frame $\{e_1,\dots , e_{2p}\}$ for $TN$ and consider
\begin{equation}\label{eqn:derivsig}
 (\mathcal{L}_v\sigma)(e_{j_1},\dots e_{j_{p+1}})=\sum_{i=1}^{p+1}(-1)^{i+1}\sigma(\nabla^\perp_{e_{j_i}}v, e_{j_1}, \dots , \hat{e}_{j_i}, \dots , e_{j_{p+1}}),
\end{equation}
where $\hat{e}_{j_i}$ means that $\sigma$ is not evaluated on this element, and we have used that $\sigma$ is parallel and vanishes when evaluated on $p+1$ tangent vectors to $N$. We notice that
\[
\Omega(e_k,Je_k,\,\cdot\, ,\dots , \, \cdot \,)=0=\overline{\Omega}(e_k,Je_k,\,\cdot\, ,\dots , \, \cdot \,).
\]
Therefore taking a frame for $TN$ of the form $\{e_1,Je_1, \dots , e_p,Je_p\}$, we see that only terms in \eqref{eqn:derivsig} of the form
\begin{align*}
(\mathcal{L}_v\sigma)(e_1,\dots ,e_p, Je_{j})&=(-1)^{j-1}\sigma(\nabla^\perp_{e_j}v, e_{1}, \dots , \hat{e}_{j}, \dots , e_p,Je_j) \\
& +(-1)^p\sigma(\nabla^\perp_{Je_j}v, e_{1}, \dots , e_p), \\
(\mathcal{L}_v\sigma)(Je_1,\dots ,Je_p, e_{j})&=(-1)^{j-1}\sigma(\nabla^\perp_{Je_j}v, Je_{1}, \dots , \widehat{Je}_{j}, \dots , Je_p, e_j) \\
& +(-1)^p\sigma(\nabla^\perp_{e_j}v, Je_{1}, \dots , Je_p) ,
\end{align*}
for $j\in \{1,\dots ,p\}$ are nonzero.

We have that $e_k-iJe_k$ is of type $(1,0)$, while $e_k+iJe_k$ is of type $(0,1)$. Therefore 
\begin{align*}
 (e_k+iJe_k)\hook \Omega=0=(e_k-iJe_k)\hook \overline{\Omega},
\end{align*}
and so we see that
\begin{align}\label{eqn:omhook}
 e_k\hook\Omega&=-iJe_k\hook\Omega, \\ \nonumber
e_k\hook \overline{\Omega}&=iJe_k\hook \overline{\Omega}.
\end{align}
So we have that
\begin{align*}
&\sigma(\nabla^\perp_{e_j}v, e_{1}, \dots , \hat{e}_{j}, \dots , e_p,Je_j)= 
\\
&\frac{i}{2}(-i)^{p-1}\Omega(\nabla^\perp_{e_j}v, Je_1,\dots , \widehat{Je_j},\dots , Je_p,e_j,\, \cdot \, ,\dots ,\,\cdot \,) \\
-&\frac{i}{2}i^{p-1}\overline{\Omega}(\nabla^\perp_{e_j}v, Je_1,\dots , \widehat{Je_j},\dots , Je_p,e_j,\, \cdot \, ,\dots ,\,\cdot \,).
\end{align*}
Calculating that
\begin{align*}
&\sigma(\nabla^\perp_{Je_j}v, e_{1}, \dots , e_p)= \\
&\frac{(-i)^p}{2}\Omega(\nabla^\perp_{Je_j}v, Je_1,\dots , Je_p, \,\cdot\, , \dots , \,\cdot \, )+\frac{i^p}{2}\overline{\Omega}(\nabla^\perp_{Je_j}v, Je_1,\dots , Je_p, \,\cdot\, , \dots , \,\cdot \, ),
\end{align*}
we find that
\begin{align*}
(*_N\mathcal{L}_v\sigma)(Je_1,\dots , \widehat{Je_j}, \dots , Je_p)&=(-1)^{\frac{p(p-1)}{2}+j-1}(\mathcal{L}_v\sigma)(e_1,\dots ,e_p, Je_{j}) \\
=\frac{i^p}{2}(-1)^{\frac{p(p-1)}{2}}(\Omega+(-1)^p\overline{\Omega})&(\nabla^\perp_{e_j}v,e_j, Je_1, \dots , \widehat{Je_j}, \dots , Je_p,\, \cdot \, ,\dots ,\,\cdot \,) \\
+\frac{i^p}{2}(-1)^{\frac{p(p-1)}{2}}(\Omega+(-1)^p\overline{\Omega})&(\nabla^\perp_{Je_j}v,Je_j, Je_1, \dots , \widehat{Je_j}, \dots , Je_p,\, \cdot \, ,\dots ,\,\cdot \,).
\end{align*}
A similar calculation yields that
\begin{align*}
(*_N\mathcal{L}_v\sigma)(e_1,\dots , \hat{e}_j, \dots , e_p)&=(-1)^{\frac{p(p+1)}{2}+j-1}(\mathcal{L}_v\sigma)(Je_1,\dots ,Je_p, e_{j}) \\
=\frac{i^p}{2}(-1)^{\frac{p(p+1)}{2}}((-1)^p\Omega+\overline{\Omega})&(\nabla^\perp_{Je_j}v,Je_j, e_1, \dots , \hat{e}_j, \dots , e_p,\, \cdot \, ,\dots ,\,\cdot \,) \\
+\frac{i^p}{2}(-1)^{\frac{p(p+1)}{2}}((-1)^p\Omega+\overline{\Omega})&(\nabla^\perp_{e_j}v,e_j, e_1, \dots , \hat{e}_j, \dots , e_p,\, \cdot \, ,\dots ,\,\cdot \,).
\end{align*}
Therefore we find that
\begin{align*}
*_N\mathcal{L}_v\sigma&=(-1)^{\frac{p(p-1)}{2}}i^p\sum_{j=1}^p\Omega(\nabla^\perp_{e_j},e_j,\, \cdot \, ,\dots ,\,\cdot \,)+\Omega(\nabla^\perp_{Je_j},Je_j,\, \cdot \, ,\dots ,\,\cdot \,) \\
&+(-1)^{\frac{p(p+1)}{2}}i^p\sum_{j=1}^p\overline{\Omega}(\nabla^\perp_{e_j},e_j,\, \cdot \, ,\dots ,\,\cdot \,)+\overline{\Omega}(\nabla^\perp_{Je_j},Je_j,\, \cdot \, ,\dots ,\,\cdot \,) \\
&=(-1)^{\frac{p(p-1)}{2}}i^p\sum_{j=1}^pe_j\hook(\nabla^\perp_{e_j}v\hook\Omega)+Je_j\hook(\nabla^\perp_{Je_j}v\hook\Omega) \\
&+(-1)^{\frac{p(p+1)}{2}}i^p\sum_{j=1}^pe_j\hook(\nabla^\perp_{e_j}v\hook\overline{\Omega})+Je_j\hook(\nabla^\perp_{Je_j}v\hook\overline{\Omega}) \\
&=(-1)^{\frac{p(p-1)}{2}}i^p\sum_{j=1}^pe_j\hook\nabla_{e_j}(v\hook\Omega)+Je_j\hook\nabla_{e_j}(v\hook\Omega) \\
&+(-1)^{\frac{p(p+1)}{2}}i^p\sum_{j=1}^pe_j\hook\nabla_{e_j}(v\hook\overline{\Omega})+Je_j\hook\nabla_{e_j}(v\hook\overline{\Omega})\\
&=i^p(-1)^{\frac{p(p-1)}{2}+1}(\partial^*(v\hook\Omega)+(-1)^p\bar\partial^*(v\hook \overline{\Omega})).
\end{align*}
The expressions for the linearisations of $G_1$ and $G_2$ follow similarly. It remains to show that if $v_1\in \nu^{1,0}_M(N)$ then
\[
 \bar\partial v_1=0 \iff \partial^*(v_1\hook\Omega)=0.
\]
Let $\{e_1,Je_1,\dots , e_p, Je_p\}$ be a unitary frame for $TN$. Then we have that
\begin{align*}
 \bar\partial v_1=0 &\iff(e^1+iJe^1)\wedge(\nabla^\perp_{e_1}+i\nabla^\perp_{Je_1})v_1+\dots \\
  &+(e^p+iJe^p)\wedge(\nabla^\perp_{e_p}+i\nabla^\perp_{Je_p})v_1 =0 \\
&\iff (\nabla^\perp_{e_1}+i\nabla^\perp_{Je_1})v_1=\dots=(\nabla^\perp_{e_p}+i\nabla^\perp_{Je_p})v_1=0 \\
&\iff [(\nabla^\perp_{e_1}+i\nabla^\perp_{Je_1})v_1]\hook \Omega=\dots =[(\nabla^\perp_{e_p}+i\nabla^\perp_{Je_p})v_1]\hook\Omega=0 \\
&\iff \Omega((\nabla^\perp_{e_1}+i\nabla^\perp_{Je_1})v_1, e_1-iJe_1,\,\cdot\,, \dots ,\,\cdot\,)+\dots  \\
&+\Omega((\nabla^\perp_{e_p}+i\nabla^\perp_{Je_p})v_1, e_p-iJe_p,\,\cdot\,,\dots ,\,\cdot\,)=0 \\
&\iff 2\sum_{i=1}^p\Omega(\nabla_{e_i}^\perp v_1,e_i,\,\cdot\, ,\dots , \,\cdot\,)+\Omega(\nabla^\perp_{Je_i}v,Je_i, \,\cdot \, , \dots ,\,\cdot \, )=0 \\
&\iff 2\sum_{i=1}^pe_i\hook \nabla_{e_i}(v_1\hook\Omega)+Je_i\hook \nabla_{Je_i}(v_1\hook\Omega)=0 \\
&\iff -2\partial^*(v_1\hook \Omega)=0,
\end{align*}
where we have exploited the property that $\Omega$ never vanishes, that $\nabla^\perp_{e_i} v$ is of type $(1,0)$ and Equation \eqref{eqn:omhook}.
\end{proof}
Similarly to Proposition \ref{prop:glin} we may identify the kernels of the operators $\bar\partial$ and $\bar\partial^*$. This will be helpful when we compare the results of this section to Kodaira's theorem \cite[Theorem 1]{MR0133841}.
\begin{cor}\label{cor:glin}
Let $N^p$ be a complex submanifold a Calabi--Yau manifold $M^m$. Consider the operators
\begin{align*}
 \bar\partial:C^\infty(\nu^{1,0}_M(N))&\to C^\infty(\Lambda^{0,1}N\otimes \nu^{1,0}_M(N)), \\
\bar\partial^*:C^\infty(\Lambda^{0,p}N\otimes\Lambda^{m-p-1}\nu^{1,0}_M(N))&\to C^\infty(\Lambda^{0,p-1}N\otimes \Lambda^{m-p-1}\nu^{1,0}_M(N)).
\end{align*}
Then there is an isomorphism
\[
 \nu^{1,0}_M(N)\to \Lambda^{0,p}N\otimes \Lambda^{m-p-1}\nu^{1,0}_M(N),
\]
that induces an isomorphism
\[
 \textnormal{Ker }\bar\partial \to \textnormal{Ker }\bar\partial^*.
\]
\end{cor}
\begin{proof}
 Let $\Omega$ be a holomorphic volume form on $M$. Then the isomorphism
\[
 \nu^{1,0}_M(N)\to \Lambda^{0,p}N\otimes\Lambda^{m-p-1} \nu^{1,0}_M(N),
\]
is given by
\[
 v\mapsto (\bar{v}\hook\overline{\Omega})^\sharp,
\]
where $\sharp:\nu^{*0,1}_M(N)\to \nu^{1,0}_M(N)$ is the musical isomorphism. The argument of Proposition \ref{prop:normal} extends to arbitrary dimensions and so this is an isomorphism. We proved in Proposition \ref{prop:glin} that
\[
 \bar\partial v=0 \iff \partial^*(v\hook \Omega)=0.
\]
Since 
\[
 \bar\partial^*(\bar{v}\hook \overline{\Omega})=\overline{\partial^*(v\hook \Omega)},
\]
the result follows.
\end{proof}

The following lemma allows us to see that the kernel of $G$ defined in Proposition \ref{prop:compdefs} is equal to the kernel of the linear part of $G$ computed in Proposition \ref{prop:glin}.

\begin{lem}\label{lem:glin}
Let $N^p$ be a complex submanifold of a Calabi--Yau manifold $M^m$. Let $v\in C^\infty(\nu_M(N)\otimes \mathbb{C})$ satisfy
\[
\partial^*(v\hook \Omega)=0=\bar\partial^*(v\hook \overline{\Omega}). 
\]
Let $G,G_1$ and $G_2$ be the operators defined in Equations \eqref{eqn:cxop}, \eqref{eqn:cxop1} and \eqref{eqn:cxop2}. Then $G(v)=0=G_1(v)=G_2(v)$.
\end{lem}
\begin{proof}
The argument here is similar to the argument of Proposition \ref{prop:caypde}.

At each point of $N$, we can write the tangent space of the (small) deformation $N_{\tilde{v}}$ as a normal graph over the tangent space of $N$. So it suffices to prove this proposition for a normal graph over a $p$-dimensional complex subspace of $\mathbb{C}^m$. Suppose that this graph is described by
\begin{align*}
 v_j&=e_j+\sum\lambda^j_i e_i, \\
w_j&=Je_j+\sum\mu^j_i e_i,
\end{align*}
where $i$ runs over $p+1$ to  $m$ and $m+p+1$ to $2m$ with $e_{m+j}=Je_j$, for $\lambda^j_i,\mu^j_i\in \mathbb{R}$.

We have that
\[
\Omega_{f(x)}=e^{i\phi}(e^1-iJe^1)\wedge \dots \wedge (e^p-iJe^p)\wedge (e^{p+1}-iJe^{p+1})\wedge \dots \wedge (e^m-iJe^m).
\]
for some $\phi\in [0,2\pi)$, where again $Je^j=-g(Je_j,\,\cdot \,)$. Write
\[
\alpha=(e^1-iJe^1)\wedge \dots \wedge (e^p-iJe^p), \quad \beta= (e^{p+1}-iJe^{p+1})\wedge \dots \wedge (e^m-iJe^m).
\]
We first find the linear terms of $\sigma=0$. We find that the only nonzero linear terms come from evaluating
\[
\sigma(v_1,\dots , v_p,w_j)=0=\sigma(w_1,\dots ,w_p,v_k),
\]
if $p+1\ne m$ and if $p+1=m$ additionally
\[
\text{Im }\Omega(v_1,\dots v_p, w_j)=0=\text{Im }\Omega(w_1,\dots , w_p, v_j).
\]
These terms are 
\begin{align*}
&\text{Re }(e^{i\phi}\alpha(v_1,\dots , v_p)\wedge (w_j\hook \beta))
\\
-&\text{Re }(e^{i\phi}\alpha(v_1,\dots , v_{j-1},w_j,v_{j+1},\dots , v_p)\wedge (v_j\hook \beta)),
\end{align*}
and additionally if $p+1=m$,
\begin{align*}
&\text{Im }(e^{i\phi}\alpha(v_1,\dots , v_p)\wedge (w_j\hook \beta)) \\
-&\text{Im }(e^{i\phi}\alpha(v_1,\dots , v_{j-1},w_j,v_{j+1},\dots , v_p)\wedge (v_j\hook \beta)).
\end{align*}
We have that
\[
\alpha(v_1,\dots v_p)=1, \quad \alpha(v_1,\dots , v_{j-1},w_j,v_{j+1},\dots v_p)=i,
\]
and so these linear terms are of the form
\begin{equation}\label{eqn:lin}
\text{Re }[e^{i\phi}(w_j\hook\beta-iv_j\hook \beta)]=0,
\end{equation}
for $j=1,\dots , p$. Notice that this implies that
\[
e^{i\phi}(w_j\hook\beta-iv_j\hook \beta)=0;
\]
indeed, if $p+1=m$ then this follows immediately. Otherwise note that
\[
(w_j\hook\beta-iv_j\hook \beta)=[(\mu^j_k+\lambda^j_{k+m})+i(\mu^j_{k+m}-\lambda^j_k)]e_k\hook\beta,
\]
and since $\text{Re }(e_k\hook \beta)\ne \text{Im }(e_k\hook \beta)$, we have that
\begin{align*}
\text{Re }[e^{i\phi}(w_j\hook\beta-iv_j\hook \beta)]=0&\iff \mu^j_k+\lambda^j_{k+m}=0=\mu^j_{k+m} -\lambda^j_k\\
 &\iff e^{i\phi}(w_j\hook\beta-iv_j\hook \beta)=0.
\end{align*}
Now given any $p+1$ of $\{v_1, \dots , v_p, w_1,\dots , w_p\}$, we must have that there is some $j_0$ so that two of these vectors are $v_{j_0}, w_{j_0}$. Notice that
\begin{align*}
\Omega(v_{j_0},w_{j_0},\, \cdot\, , \dots , \, \cdot \,)&=(-1)^{p-1}(v_{j_0}\hook \alpha)\wedge (w_{j_0}\hook \beta) +(-1)^{p}(w_{j_0}\hook \alpha)\wedge (v_{j_0}\hook \beta) \\
&+(-1)^{p}\alpha\wedge\beta(v_{j_0},w_{j_0},\, \cdot\, , \dots , \, \cdot \,).
\end{align*}
Since
\[
(v_{j_0}\hook \alpha)\wedge (w_{j_0}\hook \beta) -(w_{j_0}\hook \alpha)\wedge (v_{j_0}\hook \beta)=(e_{j_0}\hook \alpha)\wedge (w_{j_0}-iv_{j_0}\hook \beta)=0,
\]
these terms will vanish. Now
\[
\beta(v_{j_0},w_{j_0},\, \cdot\, , \dots , \, \cdot \,)=\frac{i}{2}\beta(w_{j_0}-iv_{j_0},w_{j_0}+iv_{j_0},\,\cdot \, , \dots , \, \cdot \,)=0,
\]
and so we see that \eqref{eqn:lin} vanishing implies that $\sigma$ (and $\text{Im }\Omega$) vanish on any $p+1$ tangent vectors to $N'$.
\end{proof}
We can now prove a result on the moduli space of complex deformations of a complex submanifold of a Calabi--Yau manifold.
\begin{thm}\label{thm:cxdefs}
Let $N^p$ be a compact complex submanifold of a Calabi--Yau manifold $M^m$. Then the moduli space of complex deformations of $N$ in $M$ is a smooth manifold of dimension 
\[
2 \,\textnormal{dim}_{\mathbb{C}} \textnormal{Ker }\bar\partial=\textnormal{dim}_\mathbb{R}\bar\partial,
\]
where 
\[
\bar\partial:C^\infty(\nu^{1,0}_M(N))\to C^\infty(\Lambda^{0,1}N\otimes \nu^{1,0}_M(N)).
\]
\end{thm}
\begin{proof}
By Proposition \ref{prop:comppde} we can identify the moduli space with the kernel of the partial differential operator \eqref{eqn:cxop}, or kernels of the partial differential operators \eqref{eqn:cxop1} and \eqref{eqn:cxop2}. By Lemma \ref{lem:glin} the kernel of each of these operators is the same as the kernel of the operator
\begin{align*}
v\mapsto \bar\partial^*(v\hook \overline{\Omega})\pm\partial^*(v\hook \Omega),
\end{align*}
where $v\in C^\infty(\nu_M(N)\otimes \mathbb{C})$. By Proposition \ref{prop:glin} and Corollary \ref{cor:glin}, both the set of $v\in C^\infty(\nu_M(N)\otimes\mathbb{C})$ so that 
\[
\partial^*(v\hook \Omega)=0,
\]
and the set of $v\in C^\infty(\nu_M(N)\otimes\mathbb{C})$ so that 
\[
\bar\partial^*(v\hook \overline{\Omega})=0,
\]
can be identified with the kernel of
\[
\bar\partial:C^\infty(\nu^{1,0}_M(N))\to C^\infty(\Lambda^{0,1}N\otimes \nu^{1,0}_M(N)).
\]
\end{proof}
We can now deduce a result about Cayley deformations of a compact complex surface inside a Calabi--Yau four-fold.
\begin{thm}\label{thm:maincomp}
 Let $N$ be a two-dimensional compact complex submanifold of a four-dimensional Calabi--Yau manifold $M$. Then the moduli space of Cayley deformations of $N$ in $M$ is the moduli space of complex deformations of $N$ in $M$, which is a smooth manifold of dimension
\[
 \textnormal{dim}_\mathbb{C} \textnormal{Ker }\bar\partial+\textnormal{dim}_\mathbb{C} \textnormal{Ker }\bar\partial^*=2 \,\textnormal{dim}_{\mathbb{C}} \textnormal{Ker }\bar\partial,
\]
where 
\begin{align*}
 \bar\partial:C^\infty(\nu^{1,0}_M(N))&\to C^\infty(\Lambda^{0,1}N\otimes \nu^{1,0}_M(N)), \\
\bar\partial^*:C^\infty(\Lambda^{0,2}N\otimes\nu^{1,0}_M(N))&\to C^\infty(\Lambda^{0,1}N\otimes \nu^{1,0}_M(N)).
\end{align*}
\end{thm}
\begin{proof}
The moduli space of complex deformations of $N$ in $M$ is a smooth manifold of the claimed dimension by Theorem \ref{thm:cxdefs}.

By Theorem \ref{thm:compactcay}, the moduli space of Cayley deformations of $N$, if it is smooth, has dimension at most equal to the dimension of the kernel of $\bar\partial+\bar\partial^*$, which is equal to the sum of the dimensions of the kernel of $\bar\partial$ and $\bar\partial^*$ since $N$ is compact. Since the moduli space of Cayley deformations of $N$ contains the moduli space of complex deformations of $N$, we see that they must have the same dimension, and therefore are the same.
\end{proof}

With this theorem, we achieve our aim of showing directly that complex and Cayley deformations of a compact complex surface inside a Calabi--Yau manifold are the same, as can be deduced from Proposition \ref{prop:calibmin}. Moreover, we have matched the result of Kodaira's theorem \cite[Theorem 1]{MR0133841} that says that the infinitesimal complex deformations of $N$ are isomorphic to the kernel of $\bar\partial$, with the improvement that we do not need to consider any obstructions.

\textbf{Acknowledgements.} I would like to thank Jason Lotay for his help, guidance and feedback on this project. I would also like to thank Alexei Kovalev, Yng-Ing Lee and Julius Ross for comments on my PhD thesis, from which this work is taken. This work was supported by the UK Engineering and Physical Sciences Research Council (EPSRC) grant EP/H023348/1 for the University of Cambridge Centre for Doctoral Training, the Cambridge Centre for Analysis.
\bibliographystyle{plain}
\bibliography{compactbib}

\end{document}